\documentclass[a4paper,twopage,reqno,12pt]{amsart}
\usepackage{eurosym}
\usepackage[top=30mm,right=30mm,bottom=30mm,left=30mm]{geometry}
\usepackage{tabularx}
\usepackage{graphicx}
\usepackage{amsmath}
\usepackage{amssymb}
\usepackage{amsfonts}
\usepackage{amsthm}
\usepackage{amstext}
\usepackage{amsbsy}
\usepackage{amsopn}
\usepackage{amscd}
\usepackage{enumerate}
\usepackage{color}
\usepackage[colorlinks]{hyperref}
\usepackage[notref,notcite]{showkeys}
\usepackage{multirow}
\usepackage{longtable}
\usepackage{float}
\usepackage{lscape}
\usepackage{pdflscape}
\usepackage{url}
\usepackage{adjustbox}
\usepackage{rotating}
\usepackage{etoolbox}
\newtheorem{theorem}{Theorem}[section]

\newtheorem{lemma}[theorem]{Lemma}
\newtheorem{proposition}[theorem]{Proposition}

\AtEndEnvironment{proof}{\setcounter{claim}{0}}
\theoremstyle{definition}

\newtheorem{example}[theorem]{Example}
\newtheorem{remark}[theorem]{Remark}

\numberwithin{equation}{section}

\renewcommand{\leq}{\leqslant}
\renewcommand{\geq}{\geqslant}

\begin{document}

\title[A Classification of the flag-transitive $2$-$(v,3,\lambda)$ designs]{A Classification of the flag-transitive $2$-$(v,3,\lambda)$ designs with with $v\equiv 1,3\pmod{6}$ and $v \equiv 6 \pmod{\lambda}$} 
\author[]{Alessandro Montinaro and Eliana Francot}
\address{Eliana Francot, Dipartimento di Matematica e Fisica “E. De Giorgi”, University of Salento, Lecce, Italy}
\email{alessandro.montinaro@unisalento.it}
\address{Alessandro Montinaro, Dipartimento di Matematica e Fisica “E. De Giorgi”, University of Salento, Lecce, Italy}
\email{alessandro.montinaro@unisalento.it}
\subjclass[MSC 2020:]{05B05; 05B25; 20B25}
\date{\today }

\begin{abstract}
In this paper, we provide a complete classification of the $2$-$(v,3,\lambda )$ designs with $v\equiv 1,3\pmod{6}$ and $%
v \equiv 6 \pmod{\lambda}$ admitting a flag-transitive automorphism group non-isomorphic to a subgroup of $A\Gamma L_{1}(v)$. 
\end{abstract}

\maketitle

\section{Introduction and main result}\label{IandR}
A $2$-$(v,k,\lambda )$ \emph{design} $\mathcal{D}$ is a pair $(\mathcal{P},%
\mathcal{B})$ with a set $\mathcal{P}$ of $v$ points and a set $\mathcal{B}$
of blocks such that each block is a $k$-subset of $\mathcal{P}$ and each two
distinct points are contained in $\lambda $ blocks. We say $\mathcal{D}$ is 
\emph{non-trivial} if $2<k<v$, and symmetric if $v=b$. All $2$-$(v,k,\lambda
)$ designs in this paper are assumed to be non-trivial. An automorphism of $%
\mathcal{D}$ is a permutation of the point set which preserves the block
set. The set of all automorphisms of $\mathcal{D}$ with the composition of
permutations forms a group, denoted by $\mathrm{Aut(\mathcal{D})}$. For a
subgroup $G$ of $\mathrm{Aut(\mathcal{D})}$, $G$ is said to be \emph{%
point-primitive} if $G$ acts primitively on $\mathcal{P}$, and said to be 
\emph{point-imprimitive} otherwise. In this setting, we also say that $%
\mathcal{D}$ is either \emph{point-primitive} or \emph{point-imprimitive}, respectively. A \emph{flag} of $\mathcal{D}$ is a pair $(x,B)$ where $x$ is
a point and $B$ is a block containing $x$. If $G\leq \mathrm{Aut(\mathcal{D})%
}$ acts transitively on the set of flags of $\mathcal{D}$, then we say that $%
G$ is \emph{flag-transitive} and that $\mathcal{D}$ is a \emph{%
flag-transitive design}.

The $2$-$(v,k,\lambda )$ designs $\mathcal{D}$ admitting a flag-transitive
automorphism group $G$ have been widely studied by several authors. If $\lambda=1$, that is when $\mathcal{D}$ is a linear space, then $G$ acts point-primitively on $\mathcal{D}$ by a famous result due to Higman and McLaughlin \cite{HM} dating back to 1961. In 1990, Buekenhout, Delandtsheer, Doyen, Kleidman, Liebeck and Saxl \cite{BDDKLS} obtained a classification of $2$-designs with $\lambda =1$ except when $v$ is a power of a prime and $G \leq A \Gamma L_{1}(v)$. If $\lambda>1$, it is no longer true that $G$ acts point-primitively on $\mathcal{D}$ as shown by Davies in \cite{Da}.

It is proven in \cite{Mo} that, the blocks of imprimitivity of a family of flag-transitive point-imprimitive symmetric $2$-designs investigated in \cite{PZ} have the structure of the flag-transitive $2$-$(v,3,\lambda )$-designs with $v\equiv 1,3\pmod{6}$ and $v \equiv 6 \pmod{\lambda}$. Thus, the purpose of this paper is to establish a classification of such $2$-design in order to provide meaniningful constraints on the above mentioned family of flag-transitive point-imprimitive symmetric $2$-designs investigated in \cite{PZ}. Here, it is the classification:

\bigskip

\begin{theorem}
\label{Main}Let $\mathcal{D}$ be a $2$-$(v,3,\lambda )$-design with $v\equiv
1,3\pmod{6}$ and $v \equiv 6 \pmod{\lambda}$ admitting a flag-transitive
automorphism group $G$. Then one of the following holds:

\begin{enumerate}

\item Then $\mathcal{D}$ is $2$-$(p^{d},3,\lambda )$ design with either $p=3$ and $d$ odd, or $p^{d}\equiv 7\pmod{12}$, and $G_{0}<\Gamma L_{1}(p^{d})$.

\item $\mathcal{D}$ is a $2$-design as in Example \ref{Ex}(1) for $q$ even, $q-1\mid h-6$ and either $h\equiv 0 \pmod{3}$ for $q\equiv 1 \pmod{3}$, or $q\equiv 2 \pmod{3}$. Moreover, one of he following holds:

\begin{enumerate}
\item $PSL_{h}(q)\trianglelefteq G\leq P\Gamma L_{h}(q)$.

\item $G\cong A_{7}$ and $(h,q)=(4,2)$.
\end{enumerate}

\item $\mathcal{D}$ is a $2$-design as in Example \ref{Ex}(2) for $q=5$ and $h$ odd, and $PSL_{h}(5)\trianglelefteq G\leq PGL_{h}(5)$.

\item $\mathcal{D}\cong AG_{h}(3)$, $h \geq 2$, and $G$ is of affine type.
\end{enumerate}
\end{theorem}

\bigskip

\begin{remark}
Additional constrains for $\mathcal{D}$ and $G$ as in case (1) are provided in Proposition \ref{tre2tr}. Examples corresponding to (1) are the Netto triple systems, in particular $PG_{2}(2)$, for $p^{d}\equiv 7\pmod{12}$, and $\mathcal{D}\cong AG_{h}(3)$ for $p=3$. More details on Netto triple systems can be found in \cite[Section 4.2]{BDDKLS}.

In cases (2)--(4), $G$ acts point-$2$-transitively on $\mathcal{D}$.
\end{remark}

\bigskip

\begin{example}\label{Ex}
Let $\mathcal{P}$ be the point set of $PG_{h-1}(q)$, $h \geq 2$, $\mathcal{B}_{1}$ the set of all unordered triples of collinear points of $PG_{h-1}(q)$ and $\mathcal{B}_{2}$ the set of all triangles of $PG_{h-1}(q)$. Then the following hold:
\begin{enumerate}
    \item $\mathcal{D}_{1}=(\mathcal{P}, \mathcal{B}_{1})$ is a $2$-$\left( \frac{q^{h}-1}{q-1},3,q-1\right)$ design admitting $PSL_{h}(q)\trianglelefteq G\leq P\Gamma L_{h}(q)$ as a flag-transitive automorphism group. Moreover, $\mathcal{D}_{1} \cong PG_{h}(2)$ for $q=2$; 
    \item $\mathcal{D}_{2}=(\mathcal{P}, \mathcal{B}_{2})$ is a $2$-$\left( \frac{q^{h}-1}{q-1},3, q^{2}\frac{q^{h-2}-1}{q-1}\right)$ design admitting $PSL_{h}(q)\trianglelefteq G\leq P\Gamma L_{h}(q)$ as a flag-transitive automorphism group.
\end{enumerate}
\end{example}

\begin{proof}
$\mathcal{D}_{i}=(\mathcal{P}, \mathcal{B}_{i})$ is a $2$-$\left( \frac{q^{h}-1}{q-1},3, \lambda_{i} \right)$ design for $i=1,2$ since $PSL_{h}(q)\trianglelefteq G\leq P\Gamma L_{h}(q)$ acts $2$-transitively $\mathcal{P}$. 

Any element $B$ of $\mathcal{B}_{1}$ is contained in a unique line $\ell$ of $PG_{h}(q)$. Then $\lambda_{1}=q-1$ since the group induced by $G_{\ell}$ on $\ell$ contains $PGL_{2}(q)$ acting $3$-transitively on the points of $\ell$ and inducing $S_{3}$ on each triple of points of $\ell$. This proves (1).

The group $G$ acts transitively on the set of all triangles of $PG_{h}(q)$ by \cite{Per}. This implies 
$$b_{2}={h\brack 3}_{q}\frac{1}{6}q^{3}\left( q+1\right) \left( q^{2}+q+1\right) =\frac{q^{3}(q^{h}-1)(q^{h-1}-1)(q^{h-2}-1)}{6(q-1)^{3}} \textit{,}$$
and hence
$$r_{2} =\frac{q^{3}(q^{h-1}-1)(q^{h-2}-1)}{2(q-1)^{2}}\textit{ and }\lambda
_{2}=\frac{q^{2}(q^{h-2}-1)}{(q-1)}\textit{.}$$  
Clearly, $G$ acts flag-transitively on $\mathcal{D}_{i}$ for $i=1,2$.
\end{proof}

\bigskip

\begin{lemma}
\label{PP}If $\mathcal{D}$ is a $2$-$(v,3,\lambda )$ design with $v\equiv 1,3\pmod{6}$ and $v \equiv 6 \pmod{\lambda}$ admitting a flag-transitive automorphism group $G$, then one
the following holds:

\begin{enumerate}
\item $G$ acts point-$2$-transitively on $\mathcal{D}$.

\item $G$ is a rank $3$ automorphism group of $\mathcal{D}$ and for any point $x$ of $\mathcal{D}$ the set $\mathcal{P}\setminus \left\lbrace
x \right\rbrace$ is split into two $G_{x}$-orbits each of length $r/\lambda=(v-1)/2$.

\end{enumerate}
\end{lemma}

\begin{proof}
Either (1) holds, or $G$ is a point-primitive rank $3$ automorphism group of $\mathcal{D}$ and for any point $x$ of $\mathcal{D}$ the set $\mathcal{P}\setminus \left\lbrace
x \right\rbrace$ is split into two $G_{x}$-orbits each of length $(v-1)/2$ by \cite[Corollary 4.6]{Ka69} since $(v-1,k-1)\leq 2$, and hence $r/\lambda=(v-1)/2$.
\end{proof}

\bigskip

\begin{proposition}
\label{tre2tr}If $G$ is a rank $3$ automorphism group of $\mathcal{D}$, then the following hold:
\begin{enumerate}
    \item $\mathcal{D}$ is $2$-$(p^{d},3,\lambda )$ design with either $p=3$ and $d$ odd, or $p^{d}\equiv 7\pmod {12}$;
    \item There are positive integers $m,e,s$ such that $m$ is odd, $e$ is
even, $(m,e)=1$, $ms\mid d$ the prime divisors of $m$ divide $p^{s}-1$ and the following hold:
\begin{enumerate}
\item $G_{0}=\left\langle \bar{\omega}^{2m},\bar{\omega}^{e}\alpha
^{s}\right\rangle $, where $\omega $ be primitive element of $%
GF(p^{d})^{\ast }$, $\bar{\omega}:x\mapsto \omega x$ and $\alpha
:x\mapsto x^{p}$;

\item $B=\left\{ 0,\omega ^{2ms},\omega ^{2mt+1}\right\} $ for some $0\leq
s,t<p^{d}-1$; 

\item $\lambda =\frac{d}{ms\left\vert G_{0,B}\right\vert }$ divides $p^{d}-6$.
\end{enumerate}   
\end{enumerate}
\end{proposition}

\begin{proof}
Assume that $G$ is a rank $3$ automorphism group of $\mathcal{D}$. Then $G$ acts point-primitively on $\mathcal{D}$ by \cite[2.3.7(e)]{Demb}, and hence either $Soc(G)$ is non-abelian simple, or $Soc(G)$ is an elementary abelian $p$-group for some prime $p$ by \cite[Theorem 1.1]{BGLPS}.
Assume that $Soc(G)$ is non-abelian simple, then $%
A_{7}\trianglelefteq G\leq S_{7}$ and $v =21$ by \cite[Theorem 1.2]{BGLPS}. Then $\mathcal{D}$ is a $2
$-$(21,3,\lambda )$ design with $\lambda \mid 15$. Then $\lambda =1$ or $3$ since $%
r=10\lambda $ must divide the order of $G$. Actually, $\lambda =1$ is ruled
out in \cite{Del}. Thus, $\lambda=3$ and $r=30$.

Let $x$ be any point of $\mathcal{D}$, then either $G_{x}\cong S_{5}$ or $G_{x}\cong S_{5}\times Z_{2}$ according to
whether $G\cong A_{7}$ or $S_{7}$, respectively, by \cite{At}. Hence, if $B$ is
any block of $\mathcal{D}$ incident with $x$,  either $\left\vert
G_{x,B}\right\vert =4$ or $\left\vert G_{x,B}\right\vert =8$, respectively, since $r=30$. Moreover, $G_{x,B}$ fixes $B$ pointwise
since $B$ intersects each $G_{x}$-orbit in exactly one point of $\mathcal{D}$. However, this
is impossible since $G_{x,y}\cong S_{3}$ or $S_{3}\times Z_{2}$ for any
point of $y$ distinct from $x$ according to
whether $G\cong A_{7}$ or $S_{7}$, respectively, since the non-trivial $G_{x}$-orbits have length $10$.

Assume that $Soc(G)$ is an elementary abelian $p$-group for some prime $p$. Denote $Soc(G)$ by $T$. We may identify the point set of $\mathcal{D}$ with a $h$-dimensional $GF(p)$-space $V$ in a way that $T$ is the translation group of $V$ and $G=T:G_{0}$ with $G_{0} \leq GL_{d}(p)$ acting irreducibly on $V$ since $G$ acts point-primitively on $\mathcal{D}$. Therefore, $v=p^{d}$ with $p$ odd since $v \equiv 1,3 \pmod{6}$. Further, $\lambda$ is odd since $v \equiv 6 \pmod{\lambda}$ and $v$ is odd.

Assume that $G_{0} \nleq \Gamma L_{1}(p^{d})$. If $G_{0}$ is non-solvable, then $V=V_{4}(3)$, $SL_{2}(5)\trianglelefteq G_{0}$ and the non-trivial $G_{0}$-orbits have length $40$ by \cite[Tables 12--14]{Lieb2} and \cite[Corollary 1.3 and Theorems 3.10 and 5.3]{FK}, hence $-1 \in G_{0}$ in this case. Now, assume that $G_{0}$ is solvable. Then one of the following holds by \cite[Theorem 1.1 and
subsequent Remark]{F}:
\begin{enumerate}

\item[(i)] $G_{0}\leq N_{GL_{2}(7)}(Q_{8})$ and the non-trivial $G_{0}$-orbits
have length $24$;

\item[(ii)] $G_{0}\leq N_{GL_{2}(23)}(Q_{8})$ and the non-trivial $G_{0}$-orbits
have length $264$;

\item[(iii)] $G_{0}\leq N_{GL_{2}(47)}(Q_{8})$ and the non-trivial $G_{0}$-orbits
have length $1104$;
\end{enumerate}

In (i)--(iii), $8$ divides the order of $G_{0}$, hence $4$ divides
$\left\vert G_{0}\cap SL_{2}(q)\right\vert $ with $q=7,23$ or $47$. Thus $-1 \in G_{0}$ since the Sylow $2$-subgroups of $SL_{2}(q)$ for $q$ odd are generalized quaternion groups with $-1$ as their unique involution. Thus,  $-1 \in G_{0}$ in each case regardless the solubility of $G_{0}$ when $G_{0} \nleq \Gamma L_{1}(p^{q})$. 

Let $x\in \mathcal{D}$, $x\neq 0$, then $-1$ preserves the $\lambda $
blocks of $\mathcal{D}$ incident with $\pm x$. Moreover, $-1$ preserves at
least one of them, as $\lambda $ is odd. Therefore, $B=\left\{ 0,x,-x\right\} $ and hence $\left\vert G_{0,B}\right\vert =2\left\vert
G_{0,x}\right\vert $. Thus $r=\left \vert G_{0}:G_{0,B}\right\vert =\frac{\left\vert
G_{0}\right\vert }{2\left\vert G_{0,x}\right\vert }=\frac{p^{d}-1}{4}$ with $p^{d} \equiv 1 \pmod{4}$,
whereas $r=\frac{p^{d}-1}{2}\lambda $. Thus, $V=V_{4}(3)$ and $SL_{2}(5)\trianglelefteq G_{0}$ as well as (i)--(iii) are excluded.

Assume that $G_{0} \leq \Gamma L_{1}(p^{d})$. The previous argument shows that $-1 \notin G_{0}$.

Let $\omega $ be primitive element of $GF(p^{d})^{\ast }$,  $\bar{\omega}:x\mapsto \omega x$ and $\alpha
:x\mapsto x^{p}$. By \cite[Theorem 3.10]{FK}, $G_{0}=\left\langle \bar{\omega}
^{2m},\bar{\omega} ^{e}\alpha ^{s}\right\rangle $ with $2m\mid p^{d}-1$ (plus some
extra conditions on $2m$, $e$ and $s$ provided in \cite[Theorem 3.10]{FK}).
If $\frac{p^{d}-1}{2m}$ is even then $-1\in G_{0}$, a contradiction. Thus, $\frac{p^{d}-1}{2m}$ is odd. Note that, $%
2m\mid e\frac{p^{d}-1}{p^{s}-1}$ by \cite[see the proof of Theorem 3.11]%
{FK}, and hence $p^{s}-1\mid e\frac{p^{d}-1}{2m}$.

Suppose that $m$ is even.
Then either $e$ is odd, or $e\equiv 2\pmod{4}$ and $p^{s}\equiv 1\pmod{4}$ by \cite[Theorem 3.10(4)]{FK}. Thus $%
\frac{p^{d}-1}{2m}$ is even, a contradiction. Therefore, $m$ is odd, $e$ is even, $%
(m,e)=1$, $ms\mid d$ and primes of $m$ divide $p^{s}-1$ again by \cite[Theorem
3.10]{FK}. In particular, $p^{d}\equiv 3\pmod{4}$. Then either $p=3$
and $d$ odd, or $p^{d}\equiv 7\pmod{12}$ since $p^{d}\equiv 1,3\pmod{6}$. This proves (1) and (2.a).

The two non-trivial $G_{0}$-orbits on $GF(p^{d})$, say $y_{i}^{G_{0}}$ $i=1,2$ are the set of the squares and non-squares by (2.a) since $e$ is even. Hence, if $B$ is any
element then there are $0\leq s,t<p^{d}-1$ such that $B=\left\{ 0,\omega
^{2ms},\omega ^{2mt+1}\right\} $ since $\left\vert B\cap
y_{i}^{G_{0}}\right\vert =1$ for $i=1,2$, which is (2.b). Moreover, $%
G_{0,B}=G_{0,\omega ^{2ms},\omega ^{2mt+1}}\leq G_{0,\omega ^{2ms}}$.   

Finally, $r=\frac{(p^{d}-1)d}{2ms\left\vert G_{0,B}\right\vert }$ since $%
\left\vert G_{0}\right\vert =\frac{(p^{d}-1)d}{2ms}$, and hence $\lambda =%
\frac{d}{ms\left\vert G_{0,B}\right\vert }$ divides $p^{d}-6$ since $r=\frac{p^{d}-1}{2}%
\lambda $ and $p^{d}\equiv \lambda \pmod{6}$. Thus, we obtain (2.c).
\end{proof}

\bigskip
From now on, we assume that $G$ acts point-$2$-transitively on $\mathcal{D}$, which is the remaining case of Lemma \ref{PP} to be analyzed. In this case, $G$ is either almost simple or of affine type.  
\bigskip

\begin{theorem}
\label{treAS}If $G$ is almost simple, then one of the following holds:

\begin{enumerate}
\item $\mathcal{D}$ is a $2$-design as in Example \ref{Ex}(1) for $q$ even, $q-1\mid h-6$ and either $h\equiv 0 \pmod{3}$ for $q\equiv 1 \pmod{3}$, or $q\equiv 2 \pmod{3}$. Moreover, one of he following holds:

\begin{enumerate}
\item $PSL_{h}(q)\trianglelefteq G\leq P\Gamma L_{h}(q)$;

\item $G\cong A_{7}$ and $(h,q)=(4,2)$.
\end{enumerate}

\item $\mathcal{D}$ is a $2$-design as in Example \ref{Ex}(2) for $q=5$ and $h$ odd, and $PSL_{h}(5)\trianglelefteq G\leq PGL_{h}(5)$.
\end{enumerate}
\end{theorem}

\begin{proof}
Assume that $G$ is an almost simple point-$2$-transitive automorphism group
of $\mathcal{D}$. Then one of the following holds by \cite[(A)]{Ka85} since $v\equiv 1,3%
\pmod{6}$:

\begin{enumerate}
\item[(i)] $Soc(G)\cong A_{v}$ and $\lambda \mid v-6$, $v\equiv 1,3\pmod{6}$;

\item[(ii)] $Soc(G)\cong PSL_{h}(q)$, $h \geq 2$, $\frac{%
q^{h}-1}{q-1}\equiv 1,3\pmod{6}$ and $(h,q)\neq (2,2)$, and $\lambda \mid \frac{q^{h}-1}{q-1}-6$;

\item[(iii)] $Soc(G)\cong PSU_{3}(2^{2m+1})$, $m>0$, and $\lambda \mid 2^{6m+3}-5$;

\item[(iv)] $Soc(G)\cong A_{7}$ and $\lambda  \mid 9$.
\end{enumerate}

If $G$ acts point-$3$-transitively on $\mathcal{D}$, then $v-2= \lambda \leq v-6$, and we reach a contradiction. Thus $G$ cannot act point-$3$-transitively on $\mathcal{D}$ and hence (i) is ruled out.

Assume that (ii) holds. If $h=2$, then the point-$2$-transitive actions of $G$ on $\mathcal{D}$ and on $PG_{1}(q)$ are equivalent. Hence, we may identify the point set of $\mathcal{D}$ with that of $PG_{1}(q)$. Also, $q$ is even since $v=q+1$ and $v \equiv 1,3 \pmod{6}$. Then $G$ acts point-$3$-transitively on $\mathcal{D}$, which is not the case. Thus $h>2$.

Now, $Soc(G)$ has two $2$-transitive
permutation representations of degree $\frac{q^{h}-1}{q-1}$, and these are the one on the set of points of $PG_{h-1}(q)$, and the other on the set of hyperplanes of $PG_{h-1}(q)$. The two conjugacy classes in $PSL_{h}(q)$ (resp. in $P \Gamma L_{h}(q)$) of the point-stabilizers and hyperplane-stabilizers are fused by a polarity of $PG_{h-1}(q)$. Thus, we may identify the point set of $\mathcal{D}$ with that of $PG_{h-1}(q)$.

Let $B$ be any block of $\mathcal{D}$. Then $\mathcal{D} \cong \mathcal{D}_{i}$, where $\mathcal{D}_{i}$ is as in Example $\ref{Ex}$ and $i=1$ or $2$ according as $B$ is contained either in a line or is a triangle of $PG_{h-1}(q)$, respectively, since $ PSL_{h}(q) \unlhd G$ with $h>2$.

If $i=1$, then $\lambda=q-1$. Now, $v \equiv 1,3 \pmod{6}$ implies $v$ odd. Further, $v \equiv 6 \pmod{\lambda}$ implies 
\begin{equation*}
\left( q-1\right) \mid \sum_{i=0}^{h-1}(q^{i}-1)+h-6 \textit{,}
\end{equation*}%
and hence $q-1 \mid v-6$ since $v=\frac{q^{h}-1}{q-1}$. If $q$ is odd then $h-6$, and hence $h$ is even. Then $v$ is even, which is a contradiction. Thus $q$ is a power of $2$.

If $q \equiv 1 \pmod{3}$, then $v \equiv h \pmod{3}$. Further, $h$ is divisible by $3$ since $q-1 \mid h-6$, and hence $v \equiv 3 \pmod{6}$ in this case.

If $q \equiv 2 \pmod{3}$, then  either $v \equiv 1 \pmod{3}$ or $v \equiv 0 \pmod{3}$ according as $h$ is odd or even, respectively. Thus we obtain we obtain (1).

If $i=2$, then $\lambda=q^{2}\frac{q^{h-2}-1}{q-1}$. Now, $v \equiv 6 \pmod{\lambda}$ implies
\begin{equation*}
q^{2}\frac{q^{h-2}-1}{q-1}\mid \frac{q^{h}-1}{q-1}-6
\end{equation*}%
and hence $q=5$. Finally, $v \equiv 1,3 \pmod{6}$ implies $h$ odd. Thus, we obtain (2).

Assume that (iii) holds. Then the actions of $G$ on the
point set of $\mathcal{D}$ and on the point set of the Hermitian
unital $\mathcal{H}$ of order $2^{2m+1}$ are equivalent. Thus, we may identify the point sets of $\mathcal{D}$ and of $\mathcal{H}$.\\ 
Let $B=\left\{ x,y,z\right\}$ be any block of $\mathcal{D}$. If $B$ is contained in a line $\ell ^{\prime }$ of $\mathcal{H}$, then
each block of $\mathcal{D}$ is contained in a unique line of $\mathcal{H}$ since $G$ transitive on the block set of $\mathcal{D}$ as well
as on the line set of $\mathcal{H}$. Thus, $\lambda =2^{2m+1}-1$ since $G$ induces a group containing $PGL_{2}(q)$ as normal subgroup, and this one acts $3$-transitively on $\ell$. So, $
2^{2m+1}-1 \mid 2^{6m+3}-5$ since $\lambda \mid v-6$, a contradiction since $m>0$. Therefore, $B$ consists of points of $\mathcal{H}$ in a triangular
configuration.

By \cite[Satz II.10.12]{Hup}, $Soc(G)_{x}=Q:C$, where $Q$ is a Sylow $2$-subgroup of $Soc(G)$ acting regularly on $\mathcal{H}\setminus \{x\}$ and $C$ is a cyclic group of order $\frac{2^{4m+2}-1}{3}$. Further, if $a$ is any line of $\mathcal{H}$ incident with $x$, then $Z(Q)$ preserves $a$ and acts regularly on $a \setminus \{x\}$, and the Frobenius group $Q/Z(Q):C$ of order $2^{4m+2}\frac{2^{4m+2}-1}{3}$ acts transitively on the set of $2^{4m+2}$ lines of $\mathcal{H}$ incident with $x$. Thus $C=Soc(G)_{x,y}$ acts semiregularly on $\mathcal{H}\setminus s$, where $s$ is the line incident with $x,y$, and hence each $G_{x,y}$-orbit in $\mathcal{H}\setminus s$ is of length a multiple $\frac{2^{4m+2}-1}{3}$. It follows that the number of blocks of $\mathcal{D}$ incident with $x,y$ is\
a multiple of $\frac{2^{4m+2}-1}{3}$. Therefore $\frac{2^{4m+2}-1}{3} \mid \lambda$, and hence $\frac{2^{4m+2}-1}{3} \mid 2^{6m+3}-5$ since $\lambda \mid v-6$ and $v=2^{6m+3}+1$. However, this is impossible since $m>0$, and hence (iii) is ruled out.   

Assume that (iv) occurs. Arguing as in (1), we may identify the points $\mathcal{D}$ with the points of $PG_{3}(2)$. Further, $%
G\cong A_{7}$, $G_{x}\cong SL_{3}(2)$ and $%
G_{x,y}\cong A_{4}$ by \cite{At}. Thus, $T\cong E_{4}$ fixes a pointwise a block $B$ of $\mathcal{D}$ incident with the points $x,y$ since $\lambda \mid 9$. Since $%
N_{G_{x}}(T)\cong S_{4}$ and $G_{x,y}\cong A_{4}$, it follows that $T$ fixes exactly two points on $%
PG_{3}(2) \setminus\{x\}$ and these are necessarily collinear with $x$ ($y$ is one of them). Thus $B$ is a line of $PG_{3}(2)$, and hence $\mathcal{D}\cong
PG_{3}(2)$, which is (1.b).
\end{proof}

\bigskip
Finally, assume that $Soc(G)$ is an elementary abelian $p$-group for some prime $p$ and denote it by $T$. We may identify the point set of $\mathcal{D}$ with a $h$-dimensional $GF(p)$-space $V$ in a way that $T$ is the translation group of $V$ and $G=T:G_{0}$ with $G_{0}$ acting irreducibly on $V$ since $G$ acts point-$2$-transitively on $\mathcal{D}$. Also, both $q$ and $\lambda$ are odd since $v=q^{n}-6$ and $v \equiv 1,3 \pmod{6}$ and $v \equiv 6 \pmod{\lambda}$. \\
\bigskip

\begin{theorem}
\label{treAff}If $G$ is of affine type, then $\mathcal{D}\cong AG_{h}(3)$.
\end{theorem}

\begin{proof}

Let $B$ be any block of $\mathcal{D}$ incident with $0$ and assume that $B$ is not contained in any $1$-dimensional $GF(p)$-subspace of 
$V$. Suppose that $-1\in G_{0}$. If $x$ is a point of $ \mathcal{D}$ and $x\neq 0$, it
follows that $-1$ preserves at least one of the $\lambda $ blocks incident with $\pm x$ since $\lambda $ is odd. Further, $-1$ contains fixes a point on $B$, which is necessarily $0$, since $k=3$. Thus, we may assume that $B$ is the block fixed by $-1$. Then $%
B=\left\{ 0,x,-x\right\} \subseteq \left\langle x\right\rangle _{GF(p)}$, as $k=3$,
but this contradicts our assumption. Therefore $-1\notin G_{0}$, and hence $%
SL_{n}(q)\trianglelefteq G_{0}$ with $n$ an odd divisor of $d$, $n \geq 3$ and $q=p^{d/n}$ by \cite[Section 2, (B)]{Ka85} and by \cite[Lemma 3.12]{BF} since $q$ is odd. 

Assume that $B\subseteq \left\langle u\right\rangle _{GF(q)}$ for some non-zero vector $u$ of $V$. Clearly, $q>p$. Now, there is an element  $\phi$ in $G_{0,\left\langle u\right\rangle _{GF(q)}}$ inducing $-1$
on $\left\langle u\right\rangle _{GF(q)}$ since the group induced by $G_{\left\langle u\right\rangle _{GF(q)}}$ on $\left\langle u\right\rangle _{GF(q)}$ contains $AGL_{1}(q)$. Therefore, $\phi $ preserves at
least one of the $\lambda $ blocks incident with $\pm u$ since $\lambda$ is odd, say $B^{\prime}$. Hence, $B^{\prime}=\left\{ 0,u,-u\right\}\subseteq  \left\langle u\right\rangle _{GF(p)}\subset \left\langle u\right\rangle _{GF(q)}$. Then $B$ is contained in a $1$-dimensional $GF(p)$-subspace of $V$ since $G$ acts-flag-transitively on $\mathcal{D}$, but is contrary to our assumption.\\
Assume that $B$ is not contained in any $1$-dimensional $GF(q)$-subspace of $V$. Then there is a $2$-dimensional $GF(q)$-subspace $\pi$ of $V$ containing $B$. Note that, $G_{0}$ acts
transitively on the set of $2$-dimensional $GF(q)$-subspaces of $V$ by \cite{Per} since $SL_{n}(q)\trianglelefteq G_{0}$. Then
\begin{equation} \label{miumiu}
{n\brack 2}_{q}\mu=\frac{%
q^{n}-1}{2}\lambda \textit{,} 
\end{equation}
where $\mu$ denotes the number of blocks of $\mathcal{D}$ containing $0$ and contained in any $2$-dimensional $GF(q)$-subspace of $V$, since $\left(B^{G_{0}},\pi^{G_{0}} \right)$ is a $1$-design by \cite[1.2.6]{Demb}. Then
\begin{equation*}
\frac{\left( q^{n}-1\right) \left( q^{n-1}-1\right) }{\left( q^{2}-1\right)
\left( q-1\right) }\mid \frac{\left( q^{n}-1\right) \left( q^{n}-6\right) }{2%
},
\end{equation*}%
since $\lambda \mid q^{n}-6$. Thus $2\left( q^{n-1}-1\right) \mid \left( q^{2}-1\right) \left(
q-1\right)(q-6)$, and hence $n = 3$ since $n$ is odd and $n \geq 3$. Now, substituting $n=3$ in (\ref{miumiu}), it results $\mu=\frac{q-1}{2}\lambda$.

Now, $SL_{3}(q)$ induces $GL_{2}(q)$ on $\pi$, and hence the group $K$ induced by $G_{\pi}$ on $\pi$ acts point-$2$-transitive on $\pi$. Thus, $\left(\pi, B^{K} \right)$ is a $2$-$(q^{2},3,\lambda^{\prime})$ design with $\lambda^{\prime} \leq \lambda$ and replication number $r'=\mu$ since any triangle is contained in a unique plane. Then $\frac{q^{2}-1}{2}\lambda^{\prime}=\mu=\frac{q-1}{2}\lambda$. Thus, $\lambda=\lambda^{\prime}(q+1)$, which is not the case since both $q$ and $\lambda$ are odd.

Finally, assume that $B\subseteq \left\langle u\right\rangle _{GF(p)}$ for some non-zero
vector $u$ of $V$. Let $H$ be the group induced by $G_{\left\langle u\right\rangle_{GF(p)} }$ on $\left\langle u\right\rangle$, then $H\cong AGL_{1}(p)$ since $G$ is a subgroup of $AGL_{d}(p)$ acting $2$-transitively on $V$. Further $\left(\left\langle u\right\rangle_{GF(p)} ,B^{H} \right)$ is a $2$-$(p,3,\lambda)$ design, and $Z_{3} \leq H_{B}\leq S_{3}$ since $G_{B} \leq G_{\left\langle u\right\rangle_{GF(p)}}$. If $H_{B} \cong Z_{3}$, then $b^{\prime}=\frac{p(p-1)}{3}$, and hence $r^{\prime}=p-1$ and $\lambda=2$, whereas $\lambda$ is odd since $\lambda \mid p^{d}-6$ and $p$ is odd. Thus $H_{B} \cong S_{3}$, and hence $p=3$ since $H\cong AGL_{1}(p)$. Therefore $B=\left\langle u\right\rangle _{GF(3)}$, and hence $\mathcal{D}\cong AG_{h}(3)$.
\end{proof}

\bigskip

\begin{proof}[Proof of Theorem \ref{Main}] The assertion follows from Proposition \ref{tre2tr} and Theorems \ref{treAS}--\ref{treAff} according as $G$ does not or does act point-$2$-transitively on $\mathcal{D}$, respectively.
\end{proof}

\end{document}